\documentclass{birkjour}

\usepackage{xcolor} 

\usepackage{graphics} 

\usepackage{amsmath}
\usepackage{amssymb}
\usepackage{amsfonts}
\usepackage{amsthm}
\usepackage{amscd}

\def \d{\mbox{\(\,\mathrm{d}\)}}

\newcommand{\eq}[1]{\overset{\text{\tiny{(#1)}}}=}
\newcommand{\sm}{\setminus}

\newcommand{\Rr}{\mathbb{R}}

\newcommand{\Qq}{\mathbb{Q}}
\newcommand{\Zz}{\mathbb{Z}}

\newcommand{\bigset}[2]{\left\{\,#1 \, \big\vert \, #2\,\right\} }

 \newtheorem{theorem}{Theorem}[section]

\newtheorem{lem}[theorem]{Lemma}
\newtheorem{cor}[theorem]{Corollary}
\theoremstyle{definition}

\theoremstyle{remark}

\numberwithin{equation}{section}

\begin{document}

%
%
%
%
%
%
%
%
%

\title[Generalized entropies and the modular group]
 {A functional equation related to generalized entropies and the modular group}

\author{Daniel Bennequin}

\address{%
 Institut de Math\'ematiques de Jussieu-Paris Rive Gauche (IMJ-PRG)\br
 Universit\'e de Paris\br
8 place Aur\'elie N\'emours\br 
 75013 Paris, France
 }
\email{bennequin@math.univ-paris-diderot.fr }

\thanks{This article was written while the second author was a graduate student at the Universit\'e Paris Diderot.}
\author{Juan Pablo Vigneaux}
\address{Max-Planck-Institute for Mathematics in the Sciences\br
Inselstra{\ss}e 22\br
04103 Leipzig, Germany 
}
\email{vigneaux@mis.mpg.de}

\subjclass{Primary 97I70, 94A17 }

\keywords{Generalized entropies, Shannon entropy, Tsallis entropy, modular group, functional equation, information cohomology}

\date{\today}

\begin{abstract}
We solve a functional equation connected to the algebraic characterization of generalized information functions. To prove the symmetry of the solution, we study a related system of functional equations, which involves two homographies. These transformations generate the modular group, and this fact plays a crucial role in  solving the system. The method suggests a more general relation between conditional probabilities and arithmetic.
\end{abstract}

\maketitle
\section{Motivation and results}

In this paper, we study the measurable solutions $u:[0,1]\to \Rr$ of the functional equation
\begin{equation}\label{functional_entropy_u}
u(1-x)+(1-x)^\alpha u\left(\frac{y}{1-x}\right) = u(y) + (1-y)^\alpha u\left(\frac{1-x-y}{1-y}\right)
\end{equation}
for all $x,y \in [0,1)$ such that $x+y\in [0,1]$. The parameter $\alpha$ can take any positive real value. 

This equation appears in the context of algebraic characterizations of information functions. Given a random variable $X$ whose range is a finite set $E_X$, a measure of its ``information content'' is supposed to be a function $f[X]: \Delta(E_X) \to \Rr$, where $\Delta(E_X)$ denotes the set of probabilities on $E_X$,
\begin{equation}
\Delta(E_X) =\bigset{p:E_X\to [0,1]}{\sum_{x\in E_X}p(x) = 1}.
\end{equation}
The most important example of such a function is the Shannon-Gibbs entropy
\begin{equation}
S_1[X](p) := -\sum_{x\in E_X} p(x)\log p(x),
\end{equation}
where $0\log 0$ equals $0$ by convention.

Shannon entropy satisfies a remarkable property, called the \emph{chain rule}, that we now describe. Let $X$ (resp. $Y$) be a variable with range $E_X$ (resp. $E_Y$);  both $E_X$ and $E_Y$ are supposed to be finite sets. The couple $(X,Y)$ takes values in a subset $E_{XY}$ of $E_X\times E_Y$, and any probability $p$ on $E_{XY}$ induce by marginalization laws $X_*p$ on $E_X$ and $Y_*p$ on $E_Y$. For instance,
\begin{equation}
X_*p(x) = \sum_{y: (x,y)\in E_{XY}} p(x,y).
\end{equation}
The chain rule corresponds to the identities
\begin{align}
S_1[(X,Y)](p) &= S_1[X](X_*p) + \sum_{x\in E_X} X_*p(x) S_1[Y](Y_*(p|_{X=x})), \\
S_1[(X,Y)](p) &= S_1[Y](Y_*p) + \sum_{y\in E_Y} Y_*p(y) S_1[X](Y_*(p|_{Y=x})), 
\end{align}
where $p|_{X=x}$ denotes the conditional probability $y\mapsto p(y,x)/X_*p(x)$. These identities reflect the third axiom used by Shannon to characterize an information measure $H$:  ``if a choice be broken down into two successive choices, the original $H$ should be the weighted sum of the individual values of $H$'' \cite{Shannon1948}.

There is a deformed version of Shannon entropy, called generalized  entropy of degree $\alpha$ \cite[Ch.~6]{Aczel1975}. For any $\alpha\in (0,\infty)\sm\{1\}$, it is defined as
\begin{equation}
S_\alpha[X](p) := \frac{1}{1-\alpha}\left(\sum_{x\in E_X} p(x)^\alpha-1\right).
\end{equation}
This function was introduced by Havrda and Charv\'at \cite{Havrda1967}. Constantino Tsallis popularized its use in physics, as the fundamental quantity of non-extensive statistical mechanics \cite{Tsallis2009}, so $S_\alpha$ is also called  Tsallis $\alpha$-entropy. It satisfies a deformed version of the chain rule:
\begin{equation}\label{deformed_chain_rule}
S_\alpha[(X,Y)](p) = S_\alpha[X](X_*p) + \sum_{x\in X} (X_*p(x))^\alpha S_\alpha[Y](Y_*(p|_{X=x})).
\end{equation}

Suppose now that, given $\alpha >0$, we want to find the most general functions $f[X]$---for a given collection of finite random variables $X$---such that
\begin{enumerate}
\item[A.]\label{assumptionA} $f[X](\delta)=0$ whenever $\delta$ is any Dirac measure|a measure concentrated on a singleton|, which means that  variables with deterministic outputs do not give (new) information when measured;
\item[B.]\label{assumptionB}  the generalized $\alpha$-chain rule holds, i.e. for any variables $X$ and $Y$ with finite range\footnote{Assumption A can be deduced from  B if one identifies $X$ with $(X,X)$ through the diagonal map $E_X \to E_X\times E_X, \;x\mapsto (x,x)$ and then evaluates \eqref{cocycle1} at $Y=X$ and $p=\delta_{x_0}$, for any $x_0\in E_X$.}
\begin{align}
f[(X,Y)](p) &= f[X](X_*p) + \sum_{x\in E_X} (X_*p(x))^\alpha f[Y](Y_*(p|_{X=x})), \label{cocycle1}\\
f[(X,Y)](p) &= f[Y](Y_*p) + \sum_{y\in E_Y} (Y_*p(y))^\alpha f[X](Y_*(p|_{Y=x})). \label{cocycle2}
\end{align}
\end{enumerate}

The simplest non-trivial case corresponds to $E_X=E_Y=\{0,1\}$ and $E_{XY}=\{(0,0),(1,0),(0,1)\}$; a probability $p$ on $E_{XY}$ is a triple $p(0,0)=a$, $p(1,0)=b$, $p(0,1)=c$, such that $X_*p=(a+c,b)$ and $Y_*p=(a+b,c)$.  The  equality between the right-hand sides of \eqref{cocycle1} and \eqref{cocycle2} reads
\begin{multline}\label{condition_on_f}
f[X](a+c,b) + (1-b)^\alpha f[Y]\left(\frac{a}{1-b},\frac{c}{1-b}\right) = \\f[Y](a+b,c)+(1-c)^\alpha f[X]\left(\frac{a}{1-c},\frac{b}{1-c}\right),
\end{multline}
for any  triple $(a,b,c)\in[0,1]^2$ such that $a+b+c=1$. Setting $a=0$ and using assumption A, we conclude that $f[X](c,1-c)=f[Y](1-c,c)=: u(c)$ for any $c\in [0,1]$. Therefore, \eqref{condition_on_f} can be written in terms of this unique unknown $u$; if moreover we set $c=y$,  $b=x$ and consequently $a=1-x-y$, we get the functional equation \eqref{functional_entropy_u}, with the stated boundary conditions.

The main result of this article is the following.

\begin{theorem}\label{unique_solution}
Let $\alpha$ be a positive real number. Suppose $u:[0,1]\to \Rr$  is a measurable function that satisfies \eqref{functional_entropy_u} for every $x,y \in [0,1)$ such that $x+y\in [0,1]$. Then, there exists $\lambda\in \Rr$ such that $u(x)=\lambda s_\alpha(x)$, where
\begin{equation*}
s_1(x) = -x \log_2 x -(1-x)\log_2 (1-x)
\end{equation*}
and
\begin{equation*}
s_\alpha(x) =\frac{1}{1-\alpha}(x^\alpha + (1-x)^\alpha -1)
\end{equation*}
when $\alpha \neq 1$.
\end{theorem}

By convention, $0\log_2 0 := \lim_{x\to 0} x \log_2 x = 0$. For $\alpha=1$, Theorem \ref{unique_solution} is essentially Lemma 2 in \cite{Kannappan1973}. Our proof depends on two independent results.
\begin{theorem}[Regularity]\label{lemma_regularity}
Any measurable solution of \eqref{functional_entropy_u} is infinitely differentiable on the interval $(0,1)$.
\end{theorem}
\begin{theorem}[Symmetry]\label{lemma_symmetry}
Any solution of \eqref{functional_entropy_u} satisfies $u(x) = u(1-x)$ for all $x\in \Qq\cap[0,1]$.
\end{theorem}
The first is proved analytically, by means of standard techniques in the field of functional equations (cf. \cite{Aczel1975,Tverberg1958,Kannappan1973}), and the second by a novel geometrical argument, relating the equation to the action of the modular group on the projective line.

 Theorems \ref{lemma_regularity} and \ref{lemma_symmetry} above imply that any measurable solution $u$ of \eqref{functional_entropy_u} must be symmetric, i.e. $u(x) = u(1-x)$ for all $x\in [0,1]$, and therefore 
\begin{equation}\label{eqn_general_entropy}
u(x) + (1-x)^\alpha u\left(\frac{y}{1-x}\right) = u(y) + (1-y)^\alpha u\left(\frac{x}{1-y}\right)
\end{equation}
whenever $x,y\in [0,1)$ and $x+y \in [0,1]$.  When $\alpha=1$, this equation is called ``the fundamental equation of information theory''; it first appeared in the work of Tverberg \cite{Tverberg1958}, who deduced it from a characterization of an ``information function'' that not only supposed a version of the chain rule, but also the invariance of the function under permutations of its arguments. Dar\'oczy introduced  the fundamental equation for general $\alpha>0$, and showed that it can be deduced from an axiomatic  characterization analogue to that of Tverberg, that again supposed invariance under permutations along with a deformed chain rule akin to \eqref{deformed_chain_rule}, see \cite[Thm.~5]{Daroczy1970}.

For $\alpha = 1$,  Tverberg \cite{Tverberg1958} showed that, if $u:[0,1]\to \Rr$ is symmetric, Lebesgue integrable and satisfies \eqref{eqn_general_entropy}, then it must be a multiple of $s_1(x)$. In \cite{Kannappan1973},  Kannappan and Ng weakened the regularity condition, showing that  all measurable solutions of \eqref{eqn_general_entropy} have the form $u(x) = As_1(x) + Bx$ (where $A$ and $B$ are arbitrary real constants), which reduces to $u(x) = As_1(x)$ when $u$ is symmetric. In fact, they solved some generalizations of the fundamental equation, proving among other things that, when $\alpha=1$, the only measurable solutions of \eqref{functional_entropy_u} are multiples of $s_1(x)$. 

For $\alpha \neq 1$, Dar\'oczy \cite{Daroczy1970} established that any $u:[0,1]\to \Rr$ that satisfies  \eqref{eqn_general_entropy} and $u(0)=u(1)$ has the form\footnote{In fact, he supposes $u(1/2)=1$, but the argument works in general.}
\begin{equation}\label{solution_u}
u(x) = \frac{u(1/2)}{2^{1-\alpha}-1} (x^\alpha + (1-x)^\alpha - 1),
\end{equation}
 \emph{without any hypotheses on the regularity of $u$.} The proof starts by proving that any solution of \eqref{eqn_general_entropy} must satisfy $u(0)=0$ (setting $x=0$), and hence be symmetric (setting $y=1-x$). Since we are able to prove symmetry of the solutions of \eqref{functional_entropy_u} \emph{restricted to rational arguments} without any regularity hypothesis, we also get the following result.

\begin{cor}
For any $\alpha\in (0,\infty)\sm \{1\}$, the only functions $u:\Qq\cap[0,1]\to \Rr$ that satisfy equation \eqref{functional_entropy_u} are multiples of $s_\alpha$. 
\end{cor}
\begin{proof}
Set $x=0$ in \eqref{functional_entropy_u} to conclude that $u(1)=0$, and $y=0$ to obtain $u(0)=0$. Moreover, $u$ must be symmetric (Theorem \ref{lemma_symmetry}), hence it must fulfill \eqref{eqn_general_entropy} when the arguments are rational. Given these facts, Dar\'oczy's proof in \cite[p.~39]{Daroczy1970} applies with no modifications when restricted to $p,q\in \Qq$.
\end{proof}

More details on the characterization of information functions by means of functional equations can be found in the classical reference \cite{Aczel1975}, which gives a detailed historical introduction. Reference \cite{Baudot2015} summarizes more recent developments in connection with homological algebra.
 
It is quite remarkable that Theorem \ref{unique_solution} serves as a fundamental result to prove that, up to a multiplicative constant, $\{S_\alpha[X]\}_{X\in \mathcal S}$ is the only collection of measurable functionals (not necessarily invariant under permutations) that satisfy the corresponding $\alpha$-chain rule, for any generic set of random variables $\mathcal S$. In order to do this, one introduces an adapted cohomology theory, called information cohomology \cite{Baudot2015}, where the chain rule corresponds to the $1$-cocycle condition and thus has an algebro-topological meaning. The details can be found in the dissertation \cite{Vigneaux2019-thesis}.

\section{The modular group}

The group $G= SL_2(\Zz)/\{\pm I\}$ is called the \textbf{modular group}; it is the image of $SL_2(\Zz)$ in $PGL_2(\Rr)$. We keep using the matrix notation for the images in this quotient. We make $G$ act on $P^1(\Rr)$ as follows: an element $$g=\begin{pmatrix}
a & b \\ c & d
\end{pmatrix}\in G$$ acting on $[x:y]\in P^1(\Rr)$ (homogeneous coordinates) gives 
$$g[x:y] = [ax + by:cx+dy].$$

Let $S$ and $T$ be the elements of $G$ defined by the matrices 
\begin{equation}
S = 
\begin{pmatrix}
0 & -1 \\
1 & 0
\end{pmatrix}\quad \text{ and }\quad
T = 
\begin{pmatrix}
1 & 1 \\
0 & 1
\end{pmatrix}.
\end{equation}
The group $G$ is generated by $S$ and $T$ \cite[Ch. VII, Th. 2]{Serre1973}; in fact, one can prove that $\langle S,T;S^2, (ST)^3\rangle$ is a presentation of $G$.

\section{Regularity: proof of Theorem \ref{lemma_regularity} }
Lemma 3 in \cite{Kannappan1973} implies that $u$ is locally bounded on $(0,1)$ and hence locally integrable. Their proof is for $\alpha =1$, but the argument applies to the general case with almost no modification, just  replacing 
$$|u(y)| = \left| u(1-x) + (1-x) u \left(\frac{y}{1-x}\right) - (1-y) u\left(\frac{1-x-y}{1-y}\right) \right| \leq 3N,$$
where $x,y$ are such that $u(1-x)\leq N$, $u \left(\frac{y}{1-x}\right) \leq N$ and $u\left(\frac{1-x-y}{1-y}\right) \leq N$, by 
$$|u(y)| = \left| u(1-x) + (1-x)^\alpha u \left(\frac{y}{1-x}\right) - (1-y)^\alpha u\left(\frac{1-x-y}{1-y}\right) \right| \leq 3N,$$
which is evidently valid too whenever $x,y\in (0,1)$.

To prove the differentiability, we also follow the method in \cite{Kannappan1973}---already present in \cite{Tverberg1958}. Let us fix an arbitrary $y_0\in (0,1)$; then, it is possible to chose $s,t\in (0,1)$, $s<t$, such that 
$$\frac{1-y-s}{1-y}, \frac{1-y-t}{1-y}\in (0,1),$$
for all $y$ in certain neighborhood of $y_0$.  We integrate \eqref{functional_entropy_u} with respect to $x$, between $s$ and $t$, to obtain
\begin{equation}\label{integrated_u}
(s-t)u(y) = \int_{1-t}^{1-s} u(x) \d x + y^{1+\alpha}\int_{\frac{y}{1-s}}^{\frac{y}{1-t}} \frac{u(z)}{z^3} \d z + (1-y)^{1+\alpha} \int_{\frac{1-y-s}{1-y}}^{\frac{1-y-t}{1-y}} u(z) \d z.
\end{equation}
The continuity of the right-hand side of \eqref{integrated_u} as a function of $y$ at $y_0$, implies that $u$ is continuous at $y_0$ and therefore on $(0,1)$. The continuity of $u$ in the right-hand side of \eqref{integrated_u} implies that $u$ is differentiable at $y_0$. An iterated application of this argument shows that $u$ is infinitely differentiable on $(0,1)$.

\section{Symmetry: proof of Theorem \ref{lemma_symmetry}}
Define the function $h:[0,1]\to \Rr$ through 
\begin{equation}
\forall x \in [0,1], \quad h(x) = u(x)-u(1-x).
\end{equation}
Observe that  $h$ is anti-symmetric around $1/2$, that is, we have
\begin{equation}\label{eq_difference_2}
\forall x\in [0,1], \quad h(x) = -h(1-x).
\end{equation}
Let now  $z\in \left[\frac 12, 1\right]$ be arbitrary and use the substitutions $x=1-z$ and $y=1-z$ in \eqref{functional_entropy_u} to derive the identity 
\begin{equation}\label{eq_difference_1}
\forall z\in\left[\frac 12, 1\right], \quad h(z) = z^\alpha h(2-z^{-1}).
\end{equation}
Using the anti-symmetry of $h$ to modify the right-hand side of the previous equation, we also deduce that
\begin{equation}\label{eq_difference_3}
\forall z \in \left[\frac 12, 1\right], \quad h(z) = - z^\alpha h (z^{-1}-1).
\end{equation}

Setting $x=0$ (respectively $y=0$) in \eqref{functional_entropy_u}, we  to conclude that $u(1)=0$ (resp.  $u(0)=0$). Hence,   the function  $h$ is subject to the boundary conditions  $h(0)=h(1)=0$. From \eqref{eq_difference_1}, it follows that $h(1/2)=h(0)/2^\alpha = 0$. If the domain of $h$ is extended to the whole real line imposing $1$-periodicity:
 \begin{equation}\label{eq_difference_periodicity}
\forall x \in ]-\infty, \infty[, \quad h(x+1) = h(x),
\end{equation}
 a similar argument can be used to determine the value of $h$ at any rational argument. To that end, it is important to establish first that \eqref{eq_difference_1} and \eqref{eq_difference_3} hold for the extended function.

\begin{theorem}\label{thm:eqns_extended_h}
The function $h$, extended periodically to $\Rr$, satisfies the equations
\begin{align}
\forall x \in \Rr, \quad & h(x) = |x|^\alpha h \left(\frac{2x-1}{x}\right), \label{extended_eq_1}\\
\forall x \in \Rr, \quad & h(x) = -|x|^\alpha h \left(\frac{1-x}{x}\right). \label{extended_eq_2}
\end{align}
\end{theorem}

We establish first the anti-symmetry around $1/2$ of the extended $h$ (Lemma \ref{aux_lemma_1}), which implies that \eqref{extended_eq_2} follows from \eqref{extended_eq_1}; the latter is a consequence of Lemmas \ref{aux_lemma_2}--\ref{last_lemma}.  

\begin{lem}\label{aux_lemma_1}
$$\forall x \in \Rr, \quad h(x) = -h(1-x).$$
\end{lem}
\begin{proof}
We write $x = [x]+\{x\}$, where $\{x\}:= x-[x]$. Then,
$$h(x) \eq{\ref{eq_difference_periodicity}}  h(\{x\}) \eq{\ref{eq_difference_2}} - h(1-\{x\})
\eq{\ref{eq_difference_periodicity}}  -h(1-\{x\}-[x]) = - h(1-x).$$ 
\end{proof}

\begin{lem}\label{aux_lemma_2}
\begin{equation}\label{auxiliary_1}
\forall x\in [1,2], \quad h(x) = x^\alpha h (2-x^{-1}).
\end{equation}
\end{lem}
\begin{proof}
For $h$ is periodic, \eqref{auxiliary_1} is equivalent to
\begin{equation}
\forall x\in [1,2], \quad h(x-1) = x^\alpha h(1-x^{-1}),
\end{equation}
 and the change of variables $u=x-1$ gives
\begin{equation}\label{auxiliary_2}
\forall u \in [0,1], \quad h(u) = (u+1)^\alpha h\left(\frac{u}{u+1}\right).
\end{equation} 
Note that $1 - \frac{u}{u+1} = \frac{1}{u+1} \in [1/2,1]$ whenever $u\in[0,1]$. Therefore,
$$h\left(\frac{u}{u+1}\right) \eq{Lemma \ref{aux_lemma_1}} -h\left(\frac{1}{u+1}\right) \eq{\ref{eq_difference_3}} \left(\frac{1}{u+1}\right)^\alpha h(u).$$
This establishes \eqref{auxiliary_2}.
\end{proof}

\begin{lem}\label{aux_lemma_3}
\begin{equation}\label{auxiliary_3}
\forall x \in[2,\infty[,  \quad h(x) = x^\alpha h(2-x^{-1}).
\end{equation}
\end{lem}
\begin{proof}
If $x\in[2,\infty[$, then $1 - \frac 1x \in \left[\frac 12, 1\right]$ and we can apply equation \eqref{eq_difference_1} to obtain
\begin{equation}\label{auxiliary_4}
h\left(1 - \frac 1x\right) \eq{\ref{eq_difference_1}}  \left(1 - \frac 1x\right)^\alpha h \left(2-\left(1 - \frac 1x\right)^{-1}\right) = \left(\frac{x-1}{x}\right)^\alpha h\left(1- \frac 1 {x-1}\right).
\end{equation}
We prove \eqref{auxiliary_3} by recurrence. The case $x\in [1,2]$ corresponds to Lemma \ref{aux_lemma_2}. Suppose it is valid on $[n-1,n]$, for certain $n\geq 2$; for $x\in [n,n+1]$,
\begin{align*}
h(x) &\eq{\ref{eq_difference_periodicity}} h(x-1) \\
&\eq{rec.} (x-1)^\alpha h(2-(x-1)^{-1}) \\
&\eq{\ref{eq_difference_periodicity}} (x-1)^\alpha h(1-(x-1)^{-1}) \\
& \eq{\ref{auxiliary_4}} x^\alpha h(1-x^{-1})\\
& \eq{\ref{eq_difference_periodicity}} x^\alpha h(1-x^{-1}).
\end{align*}
\end{proof}

\begin{lem}\label{aux_lemma_3bis}
\begin{equation}\label{auxiliary_4bis}
\forall x \in \left[0,\frac 12\right], \quad h(x) = -x^\alpha h(x^{-1}-1).
\end{equation}
\end{lem}
\begin{proof}
The previous lemma and periodicity imply that \\$h(x-1) = x^\alpha h(1-x^{-1})$ for all $x\geq 2$, i.e.
\begin{equation}\label{auxiliary_5}
\forall u \geq 1, \quad h(u) = (u+1)^\alpha h\left(1-\frac{1}{u+1}\right).
\end{equation}
Then, for $u\geq 1$,
\begin{equation}\label{auxiliary_6}
h\left(\frac{1}{u+1}\right) \eq{Lem. \ref{aux_lemma_1}} -h\left(1-\frac{1}{u+1}\right) \eq{\ref{auxiliary_5}} - \left(\frac{1}{u+1}\right)^\alpha h(u).
\end{equation}
We set $y=(u+1)^{-1}\in \left(0,\frac 12\right]$. Equation \eqref{auxiliary_6} reads
\begin{equation}
\forall y \in \left(0,\frac 12\right], \quad h(y) = -y^\alpha h(y^{-1}-1).
\end{equation}
Since $h(0) = 0$, the lemma is proved.
\end{proof}

\begin{lem}
\begin{equation}
\forall x \in \left[0,\frac 12\right], \quad h(x) = x^\alpha h(2-x^{-1}).
\end{equation}
\end{lem}
\begin{proof}
Immediately deduced from the previous lemma using the anti-symmetric property in Lemma \ref{aux_lemma_1}.
\end{proof}

\begin{lem}\label{last_lemma}
$$\forall x \in ]-\infty,0], \quad h(x) = -x^\alpha h(2-x^{-1}).$$
\end{lem}
\begin{proof}
On the one hand, periodicity implies that $h(x) = h(x+1) \eq{Lem. \ref{aux_lemma_1}} -h(1-(x+1)) = -h(-x)$. On the other, for $x\leq 0$, the preceding results  imply that $h(-x) = (-x)^\alpha h(2-(-x)^{-1})= |x|^\alpha h(2-(-x)^{-1})$. Therefore, 
\begin{align*}
h(x) &= -h(-x) = -|x|^\alpha h \left(2+ \frac{1}{x}\right) \\
& \eq{Lem. \ref{aux_lemma_1}} |x|^\alpha h\left(1-\left(2+ \frac{1}{x}\right)\right) \eq{\ref{eq_difference_periodicity}}  |x|^\alpha h \left(2- \frac{1}{x}\right).
\end{align*} 
\end{proof}

The transformations $x\mapsto \frac{2x-1}{x}$ and $x\mapsto \frac{1-x}{x}$ in Equations \eqref{extended_eq_1} and \eqref{extended_eq_2} are homographies of the real projective line $P^1(\Rr)$, that we denote respectively by $\alpha$ and $\beta$. They correspond to elements
\begin{equation}
A=
\begin{pmatrix}
2 & -1 \\
1 & 0
\end{pmatrix},
\quad
B =
\begin{pmatrix}
-1 & 1 \\
1 & 0
\end{pmatrix}
\end{equation}
in $G$, that satisfy
\begin{equation}
B^2=
\begin{pmatrix}
2 & -1 \\
-1 & 1
\end{pmatrix},
\quad
BA^{-1} =
\begin{pmatrix}
-1 & 1 \\
0 & 1
\end{pmatrix}.
\end{equation}
This last matrix corresponds to $x\mapsto 1-x$.

\begin{lem}
The matrices $A$ and $B^2$ generate $G$.
\end{lem}
\begin{proof}
Let 
$$P= S^{-1}T^{-1}=\begin{pmatrix}
0 & 1  \\
-1 & 1
\end{pmatrix}. $$
One has
\begin{equation}
P A P^{-1} = \begin{pmatrix}
1 & -1 \\
0 & 1
\end{pmatrix},
\end{equation}
and
\begin{equation}
P B^{2} P^{-1} = \begin{pmatrix}
3 & -1 \\
1 & 0
\end{pmatrix} .
\end{equation}
Therefore, $PAP^{-1} = T^{-1}$ and $S=T^{-3} P B^{-2} P^{-1}$. Inverting these relations, we obtain
\begin{equation}\label{S_and_T_presented}
T = PA^{-1}P^{-1}; \quad S=PA^3B^{-2}P^{-1}.
\end{equation}
Let $X$ be an arbitrary element of $G$. Since $Y=PXP^{-1}\in G$ and $G$ is generated by $S$ and $T$, the element $Y$ is a word in $S$ and $T$. In consequence, $X$ is a word in $P^{-1}SP$ and $P^{-1}TP$, which in turn are words $A$ and $B^2$.
\end{proof}

It is possible to find explicit formulas for $S$ and $T$ in terms of $A$ and $B^2$. Since $P=S^{-1}T^{-1}$, we deduce that $PSP^{-1}= S^{-1}T^{-1}STS$ and $PTP^{-1} = S^{-1}T^{-1}TTS = S^{-1}TS$. Hence, in virtue of \eqref{S_and_T_presented},
\begin{align*}
S & = P^{-1}S^{-1}T^{-1}STS P\\
&=(P^{-1} S^{-1}P)(P^{-1} T^{-1} P)(P^{-1} S P)(P^{-1} T P) (P^{-1} SP)\\
&= B^2 A B^{-2}A^2 B^{-2}
\end{align*}
and
\begin{align*}
T &= P^{-1} S^{-1} T S P \\
&= (P^{-1} S^{-1} P)(P^{-1} T P) (P^{-1} S P)\\
&= B^2 A^{-1} B^{-2}.
\end{align*}

To finish our proof of Proposition \ref{lemma_symmetry}, we remark that the orbit of $0$ by the action of $G$ on $P^1(\Rr)$ is $\Qq\cup\{\infty\}$, where $\Qq\cup\{\infty\}$ has been identified with $\{[p:q] \in P^1(\Rr) \mid p,q\in \Zz\}\subset P^1(\Rr)$.  This is a consequence of Bezout's identity: for every point $[p:q]\in P^1(\Rr)$ representing a reduced fraction $\frac p q \neq 0$ ($p,q \in \Zz\sm \{0\}$ and coprime), there are two integers $x,y$ such that $xq - yp = 1$. Therefore 
$$g'=\begin{pmatrix}
x & p \\
y & q
\end{pmatrix}$$
 is an element of $G$ and $g'[0:1] = [p:q]$. The case $q=0$ is covered by
 $$\begin{pmatrix}
0 & 1 \\
-1 & 0
\end{pmatrix} [0:1] = [1:0].$$

The extended equations \eqref{extended_eq_1} and \eqref{extended_eq_2} are such that
\begin{enumerate}
\item for all $x\in \Rr$,  if $h(x) = 0$ then $h(\alpha^{-1} x) = 0$ and $h(\beta^{-1} x) = 0$;
\item for all $x\in \Rr\sm\{0\}$, if $h(x) = 0$ then $h(\alpha x) = 0$ and $h(\beta x) = 0$.
\end{enumerate}
Since $h(1/2)=0$, the following lemma is the missing piece to establish that the extended $h$ vanishes on $\Qq$ (and hence the original $h$ necessarily vanishes on $[0,1]\cap \Qq)$. 
\begin{lem}
For any $r\in \Qq\sm\{0\}$, there exists a finite sequence $$w=(w_i)_{i=1}^n\in \{\alpha,\beta,\alpha^{-1},\beta^{-1}\}^n$$ such that  $r=w_n\circ   \cdots  \circ w_1(1/2)$ and, for all $i\in \{1,...,n\}$, the iterate $x_i:=w_i\circ \cdots\circ  w_1(1/2)$ does not equal $0$ or $\infty$.
\end{lem} 
\begin{proof}
Since the orbit in $P^1(\Rr)$  of $1/2$ by the group of homographies generated by $A$ and $B^2$ (i.e. $G$ itself) contains the whole set of rational numbers $\Qq$, there exists a $w$ such that $r=w_n \circ \cdots \circ w_1(1/2)$, where each $w_i$ equals $\alpha$, $\beta$ or one of their inverses. 

If some iterate  equals $0$ or $\infty$, the sequence $w$ can be modified to avoid this. Let $i\in\{0,...,n\}$ be the largest index such that $x_i\in \{0,\infty\}$; in fact, $i<n$ because $r\neq 0,\infty$.
\begin{itemize}
\item If $x_i = 0$, then  $x_{i+1} \in \{1/2,1\}$ (the possibility $x_{i+1}=\infty$ is ruled out by the choice of $i$). In the case   $x_{i+1}=1/2$, the equality $r=w_n\circ \cdots\circ  w_{i+2}(1/2)$ holds, and when $x_{i+1}=1$, we have $r=w_n\circ \cdots\circ  w_{i+2}\circ \beta(1/2)$.
\item If $x_i = \infty$, then $x_{i+1}\in \{2,-1\}$ (again, $x_{i+1}=0$ is ruled out). When $x_{i+1}=2$, we have $ r=w_n\circ  \cdots\circ  w_{i+2} \circ \beta\circ  \alpha\circ  \beta^{-1}\circ \beta^{-1}(1/2)$, and when $x_{i+1}=-1$, it also holds that $r=w_n\circ  \cdots\circ  w_{i+2} \circ  \alpha \circ \alpha \circ  \beta^{-1}\circ  \beta^{-1}(1/2)$.
\end{itemize}
\end{proof}


 \bibliographystyle{abbrv}
\bibliography{Bibliography}

\begin{thebibliography}{10}

\bibitem{Aczel1975}
J.~Acz{\'e}l and Z.~Dar{\'o}czy.
\newblock {\em On Measures of Information and Their Characterizations}.
\newblock Mathematics in Science and Engineering. Academic Press, 1975.

\bibitem{Baudot2015}
P.~Baudot and D.~Bennequin.
\newblock The homological nature of entropy.
\newblock {\em Entropy}, 17(5):3253--3318, 2015.

\bibitem{Daroczy1970}
Z.~Dar{\'o}czy.
\newblock Generalized information functions.
\newblock {\em Information and control}, 16(1):36--51, 1970.

\bibitem{Havrda1967}
J.~Havrda and F.~Charv{\'a}t.
\newblock Quantification method of classification processes. {Concept} of
  structural $ a $-entropy.
\newblock {\em Kybernetika}, 3(1):30--35, 1967.

\bibitem{Kannappan1973}
P.~Kannappan and C.~T. Ng.
\newblock Measurable solutions of functional equations related to information
  theory.
\newblock {\em Proceedings of the American Mathematical Society}, 38(2):pp.
  303--310, 1973.

\bibitem{Serre1973}
J.~Serre.
\newblock {\em A Course in Arithmetic}.
\newblock Graduate texts in mathematics. Springer, 1973.

\bibitem{Shannon1948}
C.~Shannon.
\newblock A mathematical theory of communication.
\newblock {\em Bell System Technical Journal}, 27:379--423, 623--656, 1948.

\bibitem{Tsallis2009}
C.~Tsallis.
\newblock {\em Introduction to Nonextensive Statistical Mechanics: Approaching
  a Complex World}.
\newblock Springer New York, 2009.

\bibitem{Tverberg1958}
H.~Tverberg.
\newblock A new derivation of the information function.
\newblock {\em Mathematica Scandinavica}, 6:297--298, 1958.

\bibitem{Vigneaux2019-thesis}
J.~P. Vigneaux.
\newblock {\em Topology of Statistical Systems: A Cohomological Approach to
  Information Theory}.
\newblock PhD thesis, Universit\'e Paris Diderot, 2019.

\end{thebibliography}


\end{document}